\documentclass{amsart}
\usepackage[utf8]{inputenc}
\usepackage{a4wide}
\usepackage{euscript}
\usepackage{fullpage}
\usepackage[T2A]{fontenc}
\usepackage{amsfonts}
\usepackage{amssymb, amsthm}
\usepackage{amsmath}
\usepackage{mathtools}
\usepackage{graphicx}
\usepackage{geometry}
\usepackage{tikz}
\usepackage{bbm}

\numberwithin{equation}{section}

\theoremstyle{plain}
\newtheorem{thm}{Theorem}

\newtheorem{lemma}{Lemma}

\newtheorem{corollary}{Corollary}

\numberwithin{thm}{section}
\numberwithin{lemma}{section}
\numberwithin{st}{section}
\numberwithin{corollary}{section}

\theoremstyle{definition}

\theoremstyle{remark}

\newcommand{\B}{\mathbb{B}}
\newcommand{\R}{\mathbb{R}}
\newcommand{\dd}{\mathrm{d}}

\title{Integral Identities for the boundary of a convex body}
\keywords{Blaschke--Petkantschin formula, Zähle formula, Pleijel identity, Ambartzumian-Pleijel identity, random chord}
\subjclass[2020]{primary: 60D05, 53C65; secondary: 52A22, 52B11, 26B15.}

\author[T.~Moseeva]{Tatiana Moseeva}
\address{Tatiana Moseeva, Leonhard Euler International Mathematical Institute, Russia}
\email{polezina@yandex.ru}

\begin{document}

\thanks{The work was supported by the Foundation for the
Advancement of Theoretical Physics and Mathematics ``BASIS'' and by Ministry of Science and Higher Education of the Russian Federation, agreement № 075-15-2022- 287}

\begin{abstract}
We present the multidimensional versions of the Pleijel and Ambartzumian--Pleijel identities. We also obtain the generalization of both the Blaschke--Petkantschin and Zähle formulae considering the case when some points are chosen inside the convex body and some on the boundary. Moreover,  a version of the Zähle formula for the  polytopes is derived.
\end{abstract}

\maketitle

\section{Introduction and main results}

For a convex body $K$ in $\R^d$ denote by $|K|$ its volume (the $d$-dimensional Lebesgue measure). Denote by $\partial K$ the boundary of $K$ equipped with the $(d-1)$-dimensional Hausdorff measure $\sigma_{\partial K}$.  When it causes no confusion, for abbreviation we use just $\sigma$ instead of $\sigma_{\partial K}$. Also for simplicity we write $|\partial K|$ instead of $\sigma(\partial K)$. 

For $l \in \{1, \ldots, d\}$ denote by $A_{d,l}$ the set of all affine $l$-planes in $\R^d$ equipped with the unique Haar measure $\mu_{d,l}$ invariant with respect to the rigid motions and normalized by
\begin{align}\label{2104}
    \mu_{d,l}\left(\left\{E\in A_{d,l}:\,E\cap \B^d\ne\emptyset\right\}\right)=\kappa_{d-l},
\end{align}
where $\B^k$ is  the $k$-dimensional unit ball and $\kappa_k:=|\mathbb B^k|$. Also let $\omega_k:=|\mathbb{S}^{k-1}|=k\kappa_k$, where $\mathbb{S}^{k-1}=\partial \B^k$ is the $(k-1)$-dimensional unit sphere.

Given a finite number of points $x_1,\dots,x_n\in\R^d$ denote by $[x_1,\dots,x_n]$ their convex hull. In particular, $[x_1,x_2]$ is
the segment with endpoints $x_1$ and $x_2$.

\bigskip

We start with the case $d=2$.
Let $K$ be a planar convex body with $C^1$ boundary. Consider a  line $G$ intersecting $K$ chosen with respect to the measure $\mu_{2,1}$.
In \cite{Pleijel} Pleijel discovered an identity
expressing the integral functional of the length of the chord $G\cap K$ in terms of the double integration over the boundary of $K$: 
\begin{align}\label{eq:pleijel}
    \int\limits_{G \cap K \neq \emptyset} &h(|G \cap K|) \mu_{2,1}(\dd G)  \\ \notag
    &= \frac{1}{2}\int\limits_{(\partial K)^2} h'(|x_1 - x_2|)\cos\alpha_1\cos\alpha_2\, \sigma(\dd x_1)\sigma(\dd x_2),
\end{align}
where  $h:\R^1_+\to\R^1$ is any function with continuous derivative  such that $h(0)=0$ and $\alpha_1, \alpha_2$ are angles between tangents at points $x_1,x_2$ and the chord  $[x_1,x_2]$.
Here we consider the  tangents 
lying on the same side of the chord. Moreover, Pleijel showed that
\begin{align}\label{eq:pleijelcot}
    \int\limits_{G \cap K \neq \emptyset} &h(|G \cap K|) \mu_{2,1}(\dd G)  \\ \notag 
    &= \int\limits_{G \cap K \neq \emptyset} h'(|G \cap K|)|G \cap K|\cot\alpha_1\cot\alpha_2 \mu_{2,1}(\dd G).
\end{align}
From~\eqref{eq:pleijel} it is possible to derive (see, e.g., \cite[Section 6.9]{AMS89}) an explicit form of the defect in the isoperimetric inequality:
\begin{align*}
    |\partial K|^2 - 4\pi |K| = 2\int\limits_{(\partial K)^2} \sin^2\frac{\alpha_1 - \alpha_2}{2} \sigma(\dd x_1)\sigma(\dd x_2). 
\end{align*}

In \cite{AMB}, Ambartzumian gave a combinatorial proof of the Pleijel identity and presented  its version  for the convex planar polygons which is now known as the Ambartzumian--Pleijel identity: if $P$ is a convex polygon with side lengths $a_1, \ldots, a_n$, then 
\begin{align}\label{eq:ambpleijelcot}
    &\int\limits_{G \cap P \neq \emptyset} h(|G \cap P|) \mu_{2,1}(\dd G) \\ \notag 
    &= \int\limits_{G \cap P \neq \emptyset} h'(|G \cap P|)|G \cap P|\cot\alpha_1\cot\alpha_2 \mu_{2,1}(\dd G) + \sum_{i = 1}^n \int\limits_{0}^{a_i} h(t) \mathrm{d}t.
\end{align}

There exists an analogue of the Pleijel identity for the convex bodies with  smooth boundary in $\R^3$, see \cite{AMB82}:
\begin{align*}
    \int\limits_{G \cap K \neq \emptyset}& h(|G \cap K|) \mu_{3,1}(\dd G)  \\ \notag 
    &= 4\int\limits_{(\partial K)^2} \frac{h'(|x_1 - x_2|)}{|x_1 - x_2|}\cos\alpha_1\cos\alpha_2\cos\phi_0 \sigma(\dd x_1)\sigma(\dd x_2),
\end{align*}
where $\alpha_1, \alpha_2$ are the angles between the line containing the chord $[x_1,x_2]$ and its projections onto the tangent planes, and $\phi_0$ is the angle between projections of normal vectors at points $x_1$ and $x_2$ onto the orthogonal complement of the line passing through the points $x_1,x_2$.

Our first result is the following generalization of the Pleijel identity to arbitrary dimension.
\begin{thm}\label{thm1}
Let K be a convex body in $\mathbb{R}^d$ with $C^1$ boundary. Then for every function $h: \mathbb{R} \rightarrow \mathbb{R}$ with continuous derivative such that $h(0)=0$ we have
\begin{align*}
    \int\limits_{G \cap K \neq \emptyset} &h(|G \cap K|) \mu_{d,1}(\dd G) \\ \notag  
    &= \frac{1}{(d-1)\omega_d}\int\limits_{(\partial K)^2} \frac{h'(|x_1 - x_2|)}{|x_1 - x_2|^{d-2}}\cos\alpha_1\cos\alpha_2\cos\phi_0 \sigma(\dd x_1)\sigma(\dd x_2),
\end{align*}
where  $\alpha_1,\alpha_2$ are the angles between the line containing the chord $[x_1,x_2]$  and its projections onto the tangent planes, and $\phi_0$ is the angle between projections of normal vectors at points $x_1$ and $x_2$ onto the orthogonal complement of the line passing through the points $x_1,x_2$.
\end{thm}
To formulate our further results, we need to extend some of our notation to subsets of lower dimension. Given some $E\in A_{d,l}$ and $k\in\{1,\dots,l\}$ denote by $A_{E,k}$ the set of all affine $k$-planes in $E\cong\R^l$ equipped with the unique Haar measure $\mu_{E,k}$ invariant with respect to the rigid motions in $E$ and normalized the same way as in~\eqref{2104}. Now for an arbitrary subset $M\subset\R^d$ we write
\begin{align*}
    A_{M,k}=A_{\mathrm{aff}M,k}, \quad \mu_{M,k}=\mu_{\mathrm{aff}M,k},
\end{align*}
where $\mathrm{aff}M$ denotes the affine span of $M$, i.e. the intersection of all affine planes containing $M$. We also denote by $\lambda_E$ the $l$-dimensional Lebesgue measure in $E$ and we denote $\lambda_E(|K\cap E|)$ briefly by $|K\cap E|$. Analogously, given points $x_0,\dots, x_l\in\R^d$, we write $|[x_0,\dots, x_l]|$ for the $l$-dimensional Lebesgue measure of their convex hull.

Our next result generalizes the Ambartzumian--Pleijel identity (see~\eqref{eq:ambpleijelcot}) to higher dimensions. Given a convex polytope $P$ denote by $\mathcal F(P)$ the set of its facets.
\begin{thm}\label{thm2}
Let P be a convex polytope in $\mathbb{R}^d$. Then for every function $h: \mathbb{R} \rightarrow \mathbb{R}$ with continuous derivative such that $h(0)=0$ we have
\begin{align*}
    \int\limits_{G \cap P \neq \emptyset} &h(|G \cap P|) \mu_{d,1}(\dd G) = \frac{1}{(d-1)}\int\limits_{G \cap P \neq \emptyset} h'(|G \cap P|)|G \cap P|\cot\alpha_1\cot\alpha_2\cos\phi_0 \mu_{d,1}(\dd G)  \\ \notag &+ \sum_{F\in\mathcal F(P)} \int\limits_{A_{F,1}} H(|G \cap F|) \mu_{F,1}(\mathrm{d}G),
\end{align*}
where $\alpha_1,\alpha_2$ are the angles between the line $G$ and its projections onto the tangent planes,  $\phi_0$ is the angle between projections of normal vectors at endpoints of $G \cap P$ onto the orthogonal complement of line $G$, and $H$ is the antiderivative of $h$. 
\end{thm}

The following Blaschke--Petkantschin formula  (see, e.g.,~\cite{SW08}, Theorem 7.2.7) is widely used when dealing with the convex hulls of random points in a given convex body:
if $h:(\mathbb{R}^d)^{l+1} \rightarrow \mathbb R$ is а non-negative  measurable function, $l~\in \left\{1, . . . , d \right\}$, then 
\begin{align}\label{BPF}
    &\int\limits_{(\mathbb{R}^d)^{l+1}}{h(x_0, \ldots , x_l) \mathrm{d}x_0 \ldots \mathrm{d}x_l}  \\ \notag
    &=(l!)^{d-l}b_{d,l}\int\limits_{A_{d,l}}{\int\limits_{E^{l+1}}{h(x_0, \ldots , x_l){|[x_0, \ldots, x_l]|}^{d-l}\lambda_{E}(\mathrm{d}x_0)\ldots \lambda_E(\mathrm{d}x_l)}\mu_{d,l}(\mathrm{d}E)},
\end{align}
where $b_{d,l} = \frac{\omega_{d-k+1} \cdots \omega_{d}}{\omega_1 \cdots \omega_l}$.

To prove Theorem \ref{thm1}  we will need the following Z\"ahle formula which is a counterpart of the  Blaschke--Petkantschin formula for points distributed over the boundary of the convex body (\cite{Zahle}, \cite{Reitzner}):
if $K$ is a convex body in $\R^d$ and $h:(\mathbb{R}^d)^{l+1} \rightarrow \mathbb R$ is a measurable function, $l\leq d$, then 
\begin{align}\label{Zahle}
    &\int\limits_{A_{d,l}}{\int\limits_{(E\cap  \partial K)^{l+1}}h(x_0, \ldots , x_l) 
    \mathbbm{1}_{\{x_0, \ldots, x_l \;\text{ in general position}\}} \sigma_{E\cap \partial K}(\mathrm{d}x_0)\ldots \sigma_{E\cap \partial K}(\mathrm{d}x_l)}\mu_{d,l}(\mathrm{d}E) \\ \notag
    &= \frac{1}{(l!)^{d-l}b_{d,l}}
    \int\limits_{(\partial K)^{l+1}} h(x_0, \ldots , x_l) \mathbbm{1}_{\{x_0, \ldots, x_l\; \text{ in general position}\}} \frac{1}{|[x_0, \ldots, x_l]|^{d-l}}  \\ \notag
    &\times\prod_{j=0}^{l}\|\mathrm{P}_{\mathrm{aff}(x_0, \ldots, x_l)}(n_K(x_j))\| 
    \sigma(\mathrm{d}x_0) \ldots \sigma(\mathrm{d}x_l),
\end{align}
where $n_K(x)$ denotes the outer unit normal vector to $\partial K$ at $x$ and $\mathrm{P}_E$ is the orthogonal projection onto the plane $E$.

\bigskip

Let us present the following generalisation of both the Blaschke--Petkantschin and Z\"ahle formulae for convex bodies with smooth boundary.

\begin{thm}\label{thm3}

    Let $K$ be a convex body with smooth boundary and let $h(x_0,\ldots, x_l)$ be a continuous function, $l\leq d$. Then for all $k\in\{0,\dots, l+1\}$ we have
\begin{align*}
     &\int\limits_{(\partial K)^{k}}\int\limits_{(K)^{l-k+1}}h(x_0, \ldots , x_l) \mathrm{d}x_{0} \ldots \mathrm{d}x_{l-k} 
     \sigma(\mathrm{d}x_{l-k+1}) \ldots 
     \sigma(\mathrm{d}x_l)  \\ 
    &=(l!)^{d-l}b_{d,l}\int\limits_{A_{d,l}}
    \int\limits_{(E\cap \partial K)^{k}}
    \int\limits_{(E \cap K)^{l-k+1}}h(x_0, \ldots , x_l)
    {|[x_0, \ldots, x_l]|}^{d-l} \\
    &\times \prod_{j=l-k+1}^{l}\|\mathrm{P}_E(n_K(x_j))\|^{-1} \lambda_{E}(\mathrm{d}x_0)\ldots \lambda_E(\mathrm{d}x_{l-k}) 
    \sigma_{E\cap \partial K}(\mathrm{d}x_{l-k+1})\ldots 
    \sigma_{E\cap \partial K}(\mathrm{d}x_l)\mu_{d,l}(\mathrm{d}E),
\end{align*}
where $n_K(x)$ denotes the outer unit normal vector to $\partial K$ at $x$ and $\mathrm{P}_E$ is the orthogonal projection onto the hyperplane $E$.

\end{thm}

Consider the case $k=l=1$ when we have one point on the boundary and one point inside $K$. 
Applying Theorem~\ref{thm3} to the function $h(x_0, x_1) = |x_0 - x_1|^n$ for some integer $n$ implies the following formula.
\begin{corollary}\label{2251}
For a convex body $K$ with smooth boundary we have
\[
     \int\limits_{\partial K}\int\limits_{K} |x_0 - x_1|^n \mathrm{d}x_{0} \sigma(\mathrm{d}x_{1}) = \frac{\omega_d}{4(n+d)}\int\limits_{A_{d,1}} |G \cap K|^{n+d}\left(\frac{1}{\sin \alpha_1} + \frac{1}{\sin \alpha_2}\right)\mu_{d,1}(\mathrm{d}G),
\]
where $\alpha_i$ is the angle between line $G$ and tangent hyperplane at point $x_i$.
\end{corollary}
This is a counterpart of Kingman's formula \cite{jK69}, which states that 
\begin{align*}
    \int\limits_{K^2} |x_0 - x_1|^n \mathrm{d}x_{0}\mathrm{d}x_{1} = \frac{\omega_d}{(n+d)(n+d+1)} \int\limits_{A_{d,1}} |G \cap K|^{n+d+1}\mu_{d,1}(\mathrm{d}G). 
\end{align*}
\begin{proof}[Proof of Corollary~\ref{2251}]
Applying Theorem~\ref{thm3} with $h(x_0, x_1) = |x_0 - x_1|^n$ gives
\begin{align*}
    \int\limits_{\partial K}\int\limits_{K} &|x_0 - x_1|^n \mathrm{d}x_{0} \sigma(\mathrm{d}x_{1})    \\
    &=
    \frac{\omega_{d}}{4} \int\limits_{A_{d,1}}  \left(\frac{1}{\sin \alpha_1}\int\limits_{G\cap K}|x_0 - x_1|^{n+d-1} \lambda_G(\mathrm{d}x_0) + \frac{1}{\sin \alpha_2}\int\limits_{G\cap K}|x_0 - x_2|^{n+d-1} \lambda_G(\mathrm{d}x_0) \right)\mu_{d,1}(\mathrm{d}G).
\end{align*}
Noting that
\[
    \int\limits_{G \cap K} |x_0 - x_i|^{n+d-1} \lambda_G(\mathrm{d}x_0) = \int\limits_0^{|G \cap K|} x^{n+d-1} \mathrm{d}x = \frac{|G \cap K|^{n+d}}{n+d}
\]
finishes the proof.
\end{proof}


Applying~\eqref{Zahle} to the polytopes gives the following result.
\begin{thm}\label{thm4}
Let $P$ be a convex polytope in $\mathbb{R}^d$ and let $h(x_0, \ldots, x_l)$ be a measurable function, $l \leqslant d-1$. Then 

\begin{align*}
    &\int\limits_{(\partial P)^{l+1}} h(x_0,\ldots, x_l) \sigma(\mathrm{d}x_0) \ldots \sigma(\mathrm{d}x_l)  \\
    &=(l!)^{d-l}b_{d,l} 
     \sum_{\substack{(F_0, \ldots, F_l)\in\mathcal F^{l+1}(P)\\ \exists F_i \ne F_j}}
     \int\limits_{A_{d,l}}\int\limits_{E\cap F_0}\ldots
     \int\limits_{E\cap F_l} h(x_0, \ldots, x_l)  \\
     &\times
     |[x_0, \ldots, x_l]|^{d-l} \prod_{j=0}^{l}\|\mathrm P_E(n_P(x_j))\|^{-1} \lambda_{E \cap F_0 }(\mathrm{d}x_0)\ldots \lambda_{E \cap F_l}
     (\mathrm{d}x_l)\mu_{d,l}(\mathrm{d}E)  \\
     &+ (l!)^{d-l-1}b_{d-1,l}\cdot \sum_{F\in\mathcal F(P)} 
     \int\limits_{A_{F,l}}\int\limits_{(E \cap F)^{l+1}} h(x_0, \ldots, x_l)  \\
     &\times 
     |[x_0, \ldots, x_l]|^{d-l-1}\lambda_{E\cap F}(\mathrm{d}x_0)\ldots \lambda_{E \cap F}(\mathrm{d}x_l)\mu_{F,l}(\mathrm{d}E),  
\end{align*}
where $n_P(x)$ denotes the outer unit normal vector of $P$ at $x$ and $\mathrm P_E$ is the orthogonal projection onto the hyperplane $E$.
\end{thm}

The paper is organized as follows. In Sections \ref{prf} and \ref{APprf} we present the proof of Theorem~\ref{thm1} and Theorem~\ref{thm2}, then present the proof of Theorem~\ref{thm3} in Section \ref{NewZahle}. The proof of Theorem~\ref{thm4} is presented in Section \ref{Zahlepoly}.

\section{Proof of Theorem~\ref{thm1}}\label{prf}

To prove Theorem~\ref{thm1} we need the notion of flag spaces.

Let $p, q \in \{0, \ldots, d\}$, and let $E \in A_{d, p}$ be a fixed $p$-dimensional affine subspace. 
As mentioned above,  $A_{E, q}$ is the space of all $q$-dimensional affine subspaces
contained in $E$ if $q \leqslant p$. 
If $q \geqslant p$ we denote by $A_{E, q}$ the space of all $q$-dimensional affine subspaces containing $E$. Denote by $\mu_{E,q}$ the invariant measure on $A_{E, q}$  (see \cite{SW08}, Section 7.1). 

Consider pairs of affine subspaces:
\begin{align*}
    &A(d, p, q) := \{(E,F) \in A_{d, p} \times A_{d, q} : E \subset F\},\; \text{if}\; p < q, \\
&A(d, p, q) := \{(E,F) \in A_{d, p} \times A_{d, q} : E \supset F\},\; \text{if}\; p > q.
\end{align*}

We are going to apply the following Fubini-type theorem for flag spaces (see, e.g.,~\cite{SW08}, Theorem 7.1.2):
if $0 \leqslant p < q \leqslant d-1$ and $h : A(d, p, q) \rightarrow \mathbb{R}$ is a nonnegative measurable function, then

\begin{align}\label{eq:flags}
    &\int\limits_{A_{d,q}}\int\limits_{A_{F,p}} h(E,F)\mu_{F,p}(\mathrm{d}E)\mu_{d,q}(\mathrm{d}F) = 
     \int\limits_{A_{d,p}}\int\limits_{A_{E,q}} h(E,F)\mu_{E,q}(\mathrm{d}F)\mu_{d,p}(\mathrm{d}E).
\end{align}

\bigskip

Now let us prove Theorem~\ref{thm1}. It follows from~\eqref{Zahle} with $l=1$ that for a non-negative  measurable function $h(x_1,x_2)$
we have
\begin{align}\label{eq:zahle2p}
    \int\limits_{A_{d,1}} h(x_1,x_2) \mu_{d,1}(\mathrm{d}G) = \frac{1}{\omega_d}\int\limits_{(\partial K)^2} h(x_1,x_2)\frac{\sin\alpha_1\sin\alpha_2}{|x_1 - x_2|^{d-1}} \sigma(\mathrm{d}x_1)\sigma(\mathrm{d}x_2),
\end{align}
where $\alpha_1$ and $\alpha_2$ are the angles between line $\mathrm{aff}(x_1,x_2)$ and tangent hyperplanes at points $x_1$ and $x_2$.

Using equation \eqref{eq:flags} for $p = 1, q = 2$ we get
\begin{align*}
    \int\limits_{A_{d,1}} h(|G \cap K|)\mu_{d,1}(\mathrm{d}G) &= \int\limits_{A_{d,1}}\int\limits_{A_{G,2}} h(|G \cap K|)\mu_{G,2}(\mathrm{d}E)\mu_{d,1}(\mathrm{d}G)  \\ \notag 
    &= \int\limits_{A_{d,2}}\int\limits_{A_{E,1}} h(|G \cap K|)\mu_{E,1}(\mathrm{d}G)\mu_{d,2}(\mathrm{d}E).
\end{align*}
Applying  Pleijel identity \eqref{eq:pleijelcot} to the inner integral and planar convex body $K \cap E$ we have
\begin{align*}
    \int\limits_{A_{E,1}} h(|G \cap K|)\mu_{E,1}(\mathrm{d}G) = \int\limits_{A_{E,1}} h'(|G \cap K|)|G \cap K|\cot\psi_1\cot\psi_2 \mu_{E,1}(\mathrm{d}G),
\end{align*}
where angles $\psi_1$ and $\psi_2$ are the angles between tangents to $K \cap F$ at the endpoints of $G \cap K$ and line $G$ lying on the same side of $G$. 

Applying equation \eqref{eq:flags} once again yields
\begin{align}\label{eq:main}
    &\int\limits_{A_{d,1}} h(|G \cap K|)\mu_{d,1}(\mathrm{d}G)  \\ \notag
    &= \int\limits_{A_{d,1}}\int\limits_{A_{G,2}} h'(|G \cap K|)|G \cap K|\cot\psi_1\cot\psi_2 \mu_{G,2}(\mathrm{d}E)\mu_{d,1}(\mathrm{d}G)  \\ \notag
    &= \int\limits_{A_{d,1}} h'(|G \cap K|)|G \cap K| \left(\int\limits_{A_{G,2}} \cot\psi_1\cot\psi_2 \mu_{G,2}(\mathrm{d}E)\right)\mu_{d,1}(\mathrm{d}G)  \\ \notag 
    &= \frac{1}{\omega_d}\int\limits_{(\partial K)^2} 
    \frac{h'(|x_1 - x_2|)}{|x_1 - x_2|^{d-2}}\sin\alpha_1\sin\alpha_2\left(\int\limits_{A_{G,2}} \cot\psi_1\cot\psi_2 \mu_{G,2}(\mathrm{d}E)\right)\sigma({\mathrm{d}x_1)}\sigma({\mathrm{d}x_2)},
\end{align}
where in the last equality we used \eqref{eq:zahle2p}.

Consider the inner integral in the last expression. Note that space $A_{G,2}$ is parameterised by lines passing through the origin in the orthogonal complement of $G$. 

Denote by $u_E$ the unit vector in the orthogonal complement of $\mathrm{aff}(x_1,x_2)$ corresponding to plane $E$ (i.e. $E$ passes through $u_F$). By $t_1(E)$ and $t_2(E)$ denote unit vectors collinear with tangents  to $E \cap K$ at $x_1$ and $x_2$ and lying in the same halfplane of $E$ as $u_E$. 
Denote by $u_1$ the normalised projection of $n_{x_1}$ onto the orthogonal complement of $\mathrm{aff}(x_1,x_2)$ and by $t_1$ --- normalised projection of vector $x_2 - x_1$ onto the tangent hyperplane at $x_1$. Vectors $u_2$ and $t_2$ are defined similarly.

\begin{lemma}\label{cotlemma}
 \begin{align*}
     \cot \psi_1 = (u_1,u_E)\cdot \cot\alpha_1 \;\text{and} \; \cot \psi_2 = (u_2,u_E)\cdot \cot\alpha_2
 \end{align*}

\end{lemma}

\begin{proof}[Proof]
By definition,
\[
\cos\psi_1 = \frac{(x_2 - x_1, t_1(E)))}{|x_2 - x_1|} \;\text{and} \;\sin\psi_1 = (u_E, t_1(E))
\]

\[
\cos\alpha_1 = \frac{(x_2 - x_1, t_1)}{|x_2 - x_1|} \;\text{and} \;\sin\alpha_1 = (u_1, t_1).
\]

Hence we need to prove that 

\[
    \frac{(x_2 - x_1, t_1(E)))}{(u_E, t_1(E))} = (u_1, u_E)\cdot \frac{(x_2 - x_1, t_1)}{(u_1, t_1)}.
\]

Note that from definition of $u_1$ it follows that 
\[n_{x_1} = u_1 + \frac{x_1 - x_2}{|x_1 - x_2|}\].

Thus, 
\[  
    (x_2 - x_1, t_1) = |x_2 - x_1|\cdot(u_1 - n_{x_1}, t_1) = |x_2 - x_1|\cdot(u_1,t_1)
\]

and 

\[
    (x_2 - x_1, t_1(E)) = |x_2 - x_1|\cdot(u_1 - n_{x_1}, t_1(E)) = |x_2 - x_1|\cdot(u_1,t_1(E)).
\]

Moreover, 

\[
    t_1(E) = (t_1(E), \frac{x_2 - x_1}{|x_2 - x_1|})\cdot\frac{x_2 - x_1}{|x_2 - x_1|} + (t_1(E), u_E)\cdot u_E,
\]

hence 

\[
    (u_1, t_1(E)) = (t_1(E), u_E)\cdot (u_1,u_E) \; \text{and} \; (u_E, t_1(E)) = (t_1(E), u_E).
\]

Thus, 

\begin{align*}
    \frac{(x_2 - x_1, t_1(E)))}{(u_E, t_1(E))} &= \frac{|x_2 - x_1|\cdot(u_1,t_1(E))}{(t_1(E), u_E)} = \frac{|x_2 - x_1|(t_1(E), u_E)\cdot (u_1,u_E)}{(t_1(E), u_E)}  \\ 
    &= |x_2 - x_1|\cdot (u_1,u_E) = (u_1, u_E)\cdot \frac{(x_2 - x_1, t_1)}{(u_1, t_1)}.
\end{align*}
Proof of the second part of lemma is the same.

\end{proof}

Lemma \ref{cotlemma} yields  
\begin{align}\label{cots}
    &\int\limits_{A_{G,2}} \cot\psi_1\cot\psi_2 \mu_{G,2}(\mathrm{d}E) = \int\limits_{A_{G,2}} \cot\alpha_1\cot\alpha_2\cdot(u_1,u_E)\cdot(u_2,u_E) \mu_{G,2}(\mathrm{d}E)  \\ \notag 
    &= \cot\alpha_1\cot\alpha_2\cdot\int\limits_{\mathbb{S}^{d-2}}(u_1,u_E)\cdot(u_2,u_E)\tilde{\sigma}(\dd u_E),
\end{align}
where  $\tilde{\sigma}$ is  the uniform measure on $\mathbb{S}^{d-2}$ normalised to be probabilistic.

Substituting \eqref{cots} in \eqref{eq:main} yields
\begin{align*}
    &\int\limits_{A_{d,1}} h(|G \cap K|)\mu_{d,1}(\mathrm{d}G)  \\ \notag  &=\frac{1}{\omega_d}\int\limits_{(\partial K)^2} 
    \frac{h'(|x_1 - x_2|)}{|x_1 - x_2|^{d-2}}\cos\alpha_1\cos\alpha_2\left(\int\limits_{\mathbb{S}^{d-2}}(u_1,u_E)\cdot(u_2,u_E)\tilde{\sigma}(\dd u_E)\right)\sigma({\mathrm{d}x_1)}\sigma({\mathrm{d}x_2)},
\end{align*}

To finish the proof of the main result it suffices to prove that 

\[\int\limits_{\mathbb{S}^{d-2}}(u_1,u_E)\cdot(u_2,u_E)\tilde{\sigma}(\dd u_E) = \frac{1}{d-1}(u_1,u_2),\]

since $(u_1,u_2) = \cos(\phi_0)$. 

The latter statement follows from the fact that if point $u_F$ is uniformly distributed on the unit sphere $\mathbb{S}^{n-1}$ ($n=d-1$) than vector formed by its first $n-2$ coordinates is uniformly distributed on the unit ball $\mathbb{B}^{n-2}$.

Change the coordinate system so that first $n-2$ coordinates of $u_1$ and $u_2$ are zero. 
Then 

\begin{align}\label{prodint}
    &\int\limits_{\mathbb{S}^{n-1}}(u_1,u_E)\cdot(u_2,u_E)\tilde{\sigma}(\dd u_E) = 
    \int\limits_{\mathbb{B}^{n-2}}\int\limits_{\sqrt{1-|z|^2}\mathbb{S}^1} (x+z,u_1)\cdot(x+z,u_2)\tilde{\sigma}(\dd x)\tilde{\lambda}(\dd z)  \\\notag
    &= \int\limits_{\mathbb{B}^{n-2}}\int\limits_{\sqrt{1-|z|^2}\mathbb{S}^1} (x,u_1)\cdot(x,u_2)\tilde{\sigma}(\dd x)\tilde{\lambda}(\dd z) = 
    \int\limits_{\mathbb{B}^{n-2}}\int\limits_{\mathbb{S}^1}(1-|z|^2) (x,u_1)\cdot(x,u_2)\tilde{\sigma}(\dd x)\tilde{\lambda}(\dd z)  \\\notag
    &= \int\limits_{\mathbb{S}^1} (x,u_1)\cdot(x,u_2)\tilde{\sigma}(\dd x)\cdot \int\limits_{\mathbb{B}^{n-2}}(1-|z|^2)\tilde{\lambda}(\dd z) = 
    \frac{1}{2\pi}\int\limits_{0}^{2\pi}\cos(x)\cos(x-\phi_0) \dd x 
     \int\limits_{\mathbb{B}^{n-2}}(1-|z|^2)\tilde{\lambda}(\dd z).  
\end{align}

Note that

\begin{align}\label{cosprod}
    \int\limits_{0}^{2\pi}\cos(x)\cos(x-\phi_0) \dd x = \int\limits_{0}^{2\pi}\frac{1}{2}(\cos(\phi_0)+\cos(2x-\phi_0)) \dd x = \pi\cos(\phi_0) 
\end{align}

and 

\begin{align}\label{ballint}
    \int\limits_{\mathbb{B}^{n-2}}(1-|z|^2)\tilde{\lambda}(\dd z)& = \frac{1}{\kappa_{n-2}}\int\limits_{\mathbb{B}^{n-2}}(1-|z|^2)\lambda(\dd z) = \frac{1}{\kappa_{n-2}}\int\limits_{\mathbb{S}^{n-3}}\int\limits_0^1 (1-r^2)r^{n-3} \dd r\sigma(\dd \phi)  \\ \notag
    &=(n-2)\left(\frac{1}{n-2} - \frac{1}{n}\right) = \frac{2}{n} = \frac{2}{d-1}. 
\end{align}

Substituting \eqref{cosprod} and \eqref{ballint} in \eqref{prodint} 
 finishes the proof.
\section{Proof of Theorem~\ref{thm2}}\label{APprf}

Assume that $P$ is a convex polytope in $\mathbb{R}^d$. 
Similarly as in the proof of Theorem~\ref{thm1} we get 

\begin{align}\label{AP}
    \int\limits_{A_{d,1}} h(|G \cap P|)\mu_{d,1}(\mathrm{d}G) = \int\limits_{A_{d,2}}\int\limits_{A_{E,1}} h(|G \cap P|)\mu_{E,1}(\mathrm{d}G)\mu_{d,2}(\mathrm{d}E). 
\end{align}

Applying the Ambartzumian-Pleijel identity to the planar polygon $P \cap E$ yields
\begin{align}\label{AP1}
    \int\limits_{A_{E,1}} h(|G \cap P|)\mu_{E,1}(\mathrm{d}G) = \int\limits_{A_{E,1}} h'(|G \cap P|)|G \cap P|\cot\psi_1\cot\psi_2 \mu_{E,1}(\mathrm{d}G) +  \sum_{i = 1}^{N} \int\limits_{0}^{a_i} h(t) \mathrm{d}t,
\end{align}
where  $\psi_1$ and $\psi_2$ are the angles between tangents to $P \cap E$ at the endpoints of $G \cap P$ and line $G$ lying on the same side of $G$ and $a_i$ are the lengths of the sides of $E \cap P$. 

The second term in the right-hand side of \eqref{AP1} can be rewritten  the following way:
\begin{align}\label{AP2}
    \sum_{i = 1}^{N} \int\limits_{0}^{a_i} h(t) \mathrm{d}t  = \sum_{F\in\mathcal F(P)} \int\limits_{0}^{|E \cap F|}  h(t) \mathrm{d}t. 
\end{align}

Substituting \eqref{AP1} and \eqref{AP2} in \eqref{AP} and applying equation \eqref{eq:flags} implies 

\begin{align}\label{AP3}
     \int\limits_{A_{d,1}} h(|G \cap P|)\mu_{d,1}(\mathrm{d}G) &= \int\limits_{A_{d,1}} h'(|G \cap P|)|G \cap P| \left(\int\limits_{A_{G,2}} \cot\psi_1\cot\psi_2 \mu_{G,2}(\mathrm{d}E)\right)\mu_{d,1}(\mathrm{d}G)  \\ \notag
     &+ \sum_{F\in\mathcal F(P)} \int\limits_{A_{d,2}} \int\limits_{0}^{|E\cap F|}  h(t) \mathrm{d}t \mu_{d,2}(\mathrm{d}E). 
\end{align}

To calculate $\int\limits_{A_{G,2}} \cot\psi_1\cot\psi_2 \mu_{G,2}(\mathrm{d}E)$ we can apply Lemma~\ref{cotlemma}, which is true for $\mu_{d,2}$-almost every line $G$ intersecting $P$ (more precisely for those lines $G$ which
meet $\partial P$ in the relative interior of two different facets). 

Hence, for $\mu_{d,2}$-almost every line $G$ we have: 

\begin{align}\label{APcots}
    \int\limits_{A_{G,2}} \cot\psi_1\cot\psi_2 \mu_{G,2}(\mathrm{d}E) &= \cot\alpha_1\cot\alpha_2\cdot\int\limits_{\mathbb{S}^{d-2}}(u_1,u_E)\cdot(u_2,u_E)\tilde{\sigma}(\dd u_E) \\\notag &= \frac{1}{d-1}(u_1,u_2) \cot\alpha_1\cot\alpha_2.
\end{align}

Consider the second term of \eqref{AP3}. 
Note that for a fixed facet $F$,
\begin{align}\label{APfacets}
    \int\limits_{A_{d,2}} \int\limits_{0}^{|E\cap F|}  h(t) \mathrm{d}t \mu_{d,2}(\mathrm{d}E) = \int\limits_{A_{F,1}}\int\limits_{0}^{|G \cap F|}  h(t) \mathrm{d}t \mu_{F,1}(\mathrm{d}G) = \int\limits_{A_{F,1}} H(|G \cap F|) \mu_{F,1}(\mathrm{d}G).
\end{align}

Substituting \eqref{APcots} and \eqref{APfacets} in \eqref{AP3} yields
 \begin{align*}
     &\int\limits_{A_{d,1}} h(|G \cap P|)\mu_{d,1}(\mathrm{d}G) = \frac{1}{d-1}\int\limits_{A_{d,1}} h'(|G \cap P|)|G \cap P| \cot\alpha_1\cot\alpha_2\cos(\phi_0)\mu_{d,1}(\mathrm{d}G)  \\ &+\sum_{F\in\mathcal F(P)} \int\limits_{A_{F,1}} H(|G \cap F|) \mu_{F,1}(\mathrm{d}G), 
 \end{align*}
which finishes the proof. 
\section{Proof of Theorem~\ref{thm3}}\label{NewZahle}

The proof of Theorem~\ref{thm3} is similar to one in \cite{Reitzner}, but since it covers the case when some of the points are inside the body, we present it here. 

Consider the functional
\[
    \mathcal{I}(h) = \int\limits_{(\partial K)^{k}}\int\limits_{(K)^{l-k+1}}{h(x_0, \ldots , x_l) \prod_{j=l-k+1}^{l}\|\mathrm{P}_{\mathrm{aff}(x_0, \ldots, x_l)}(n_K(x_j))\| \mathrm{d}x_{0} \ldots \mathrm{d}x_{l-k} \sigma(\mathrm{d}x_{l-k+1}) \ldots \sigma(\mathrm{d}x_l)}.
\]

Applying Steiner's formula yields
\begin{align*}
    \mathcal{I}(h) = \lim\limits_{\varepsilon \rightarrow 0} \left(\frac{1}{\varepsilon}\right)^{k} \int\limits_{(K^{\varepsilon} \setminus K)^{k}}\int\limits_{(K)^{l-k+1}} 
    h(x_0, \ldots , x_{l-k}, \overline{x_{l-k+1}}, \ldots, \overline{x_{l}}) \prod_{j=l-k+1}^{l}\|\mathrm{P}_{\mathrm{aff}(x_0, \ldots, x_l)}(n_K(\overline{x_j}))\| \mathrm{d}x_{0} \ldots \mathrm{d}x_{l},
\end{align*}
where $K^{\varepsilon}$ is the $\varepsilon$-neighbourhood of $K$ and for a point $x \notin K$ we denote by $\overline{x}$ the point in $K$, nearest to $x$ (the metric projection). 

The Blaschke--Petkantschin formula (see~\eqref{BPF}) implies that
\begin{align*}
    \mathcal{I}(h) &= (l!)^{d-l} b_{d,l}\lim\limits_{\varepsilon \rightarrow 0} \left(\frac{1}{\varepsilon}\right)^{k} 
     \int\limits_{A_{d,l}}\int\limits_{(E \cap (K^{\varepsilon} \setminus K))^{k}}\int\limits_{(E \cap K)^{l-k+1}} 
    h(x_0, \ldots , x_{l-k}, \overline{x_{l-k+1}}, \ldots, \overline{x_{l}})  \\
    &\times \prod_{j=l-k+1}^{l}\|\mathrm{P}_{\mathrm{aff}(x_0, \ldots, x_l)}(n_K(\overline{x_j}))\|\, |[x_0, \ldots, x_l]|^{d-l} \mathrm{d}x_{0} \ldots \mathrm{d}x_{l}\mu_{d,l}(\mathrm{d}E). 
\end{align*}

Point $x$ in $E \cap (K^{\varepsilon} \setminus K)$ is determined by the point in $E \cap K$ nearest to $x$ and the distance $t(x)$ from $x$ to $E \cap \partial K$. Then $0 \leqslant t(x) \leqslant h_{E}(x)$, where $h_{E}(x)$ is the length of intersection of $K^{\varepsilon}\setminus K$ with the line in $E$ passing through $x$ and orthogonal to $E \cap \partial K$. Generaliszation of Steiner's formula applied to $E \cap (K^{\varepsilon} \setminus K)$ and continuity of integrand implies
\begin{align*}
    \mathcal{I}(h) &= (l!)^{d-l} b_{d,l}\lim\limits_{\varepsilon \rightarrow 0} \int\limits_{A_{d,l}}\int\limits_{(E \cap \partial K)^{k}}\int\limits_{(E \cap K)^{l-k+1}} h(x_0, \ldots , x_{l}) |[x_0, \ldots, x_l]|^{d-l}  \\
    &\times \prod_{j=l-k+1}^{l} \left(\frac{1}{\varepsilon}\|\mathrm{P}_{\mathrm{aff}(x_0, \ldots, x_l)}(n_K(x_j))\| \int\limits_{0}^{h_E(x_j)} \dd t\right) \lambda_{E}(\mathrm{d}x_0)\ldots \lambda_E(\mathrm{d}x_{l-k}) \\
    &\times \sigma_{E\cap \partial K}(\mathrm{d}x_{l-k+1})\ldots \sigma_{E\cap \partial K}(\mathrm{d}x_l)\mu_{d,l}(\mathrm{d}E). 
\end{align*}

Applying 
\[
    h_E(x) \leqslant \varepsilon \|\mathrm{P}_{\mathrm{aff}(x_0, \ldots, x_l)}(n_K(x))\|^{-1}
\]

and 
\[
    \lim\limits_{\varepsilon \rightarrow 0} \frac{1}{\varepsilon}h_E(x) =  \|\mathrm{P}_{\mathrm{aff}(x_0, \ldots, x_l)}(n_K(x))\|^{-1}
\]

yields 

\begin{align*}
    \mathcal{I}(h) &= (l!)^{d-l} b_{d,l}  \int\limits_{A_{d,l}}\int\limits_{(E\cap \partial K)^{k}}
    \int\limits_{(E \cap K)^{l-k+1}}h(x_0, \ldots , x_l)
    {|[x_0, \ldots, x_l]|}^{d-l} \\
    &\times
     \lambda_{E}(\mathrm{d}x_0)\ldots \lambda_E(\mathrm{d}x_{l-k}) \sigma_{E\cap \partial K}(\mathrm{d}x_{l-k+1})\ldots 
     \sigma_{E\cap \partial K}(\mathrm{d}x_l)\mu_{d,l}(\mathrm{d}E), 
\end{align*}

which finishes the proof.

\section{Proof of Theorem~\ref{thm4}}\label{Zahlepoly}

In \eqref{Zahle} we have
$\|\mathrm{P}_{\mathrm{aff}(x_0, \ldots, x_l)}(n_P(x_k))\| = 0$  if and only if all points $x_k$ lie on the same facet of $P$. 
Therefore for measurable function $h$ we have 

\begin{align*}
    &\int\limits_{(\partial P)^{l+1}} h(x_0,\ldots, x_l) \sigma(\mathrm{d}x_0) \ldots \sigma(\mathrm{d}x_l)  \\
    &= \int\limits_{(\partial P)^{l+1}} h(x_0, \ldots, x_l)\mathbbm{1}_{\{x_1, \ldots, x_l \,\text{are not all in the same facet}\}} \sigma(\mathrm{d}x_0) \ldots \sigma(\mathrm{d}x_l)  \\
    &+ \sum_{F\in\mathcal F(P)} \int\limits_{F^{l+1}} h(x_0,\ldots, x_l) \lambda_F(\mathrm{d}x_0) \ldots \lambda_F(\mathrm{d}x_l). 
\end{align*}
Applying \eqref{Zahle} to the first term and \eqref{BPF} to the second term we get the following:

\begin{align*}
     \int\limits_{(\partial P)^{l+1}} h(x_0,&\ldots, x_l) \sigma(\mathrm{d}x_0) \ldots \sigma(\mathrm{d}x_l)  \\
    & = (l!)^{d-l}b_{d,l}\int\limits_{A_{d,l}}
    \int\limits_{(E\cap  \partial P)^{l+1}} h(x_0, \ldots, x_l)
     \mathbbm{1}_{\{x_1, \ldots, x_l \,\text{are not all in the same facet}\}}  \\
     &\times |[x_0, \ldots, x_l]|^{d-l} 
    \prod_{j=0}^{l}\|\mathrm{P}_E(n_P(x_j))\|^{-1} 
     \sigma_{E\cap \partial P}(\mathrm{d}x_0)\ldots 
     \sigma_{E\cap \partial P}(\mathrm{d}x_l)\mu_{d,l}(\mathrm{d}E)  \\
     &+ (l!)^{d-l-1}b_{d-1,l}\cdot \sum_{F\in\mathcal F(P)} 
     \int\limits_{A_{F,l}}\int\limits_{(E \cap F)^{l+1}} h(x_0, \ldots, x_l)  \\
     &\times 
     |[x_0, \ldots, x_l]|^{d-l-1}\lambda_{E \cap F}(\mathrm{d}x_0)\ldots \lambda_{E \cap F}(\mathrm{d}x_l)\mu_{F,l}(\mathrm{d}E)  \\
     &= (l!)^{d-l}b_{d,l} 
     \sum_{{\substack{(F_0, \ldots, F_l)\in\mathcal F^{l+1}(P)\\ \exists F_i \ne F_j}}}
     \int\limits_{A_{d,l}}\int\limits_{E\cap F_0}\ldots
     \int\limits_{E\cap F_l} h(x_0, \ldots, x_l)  \\
     &\times
     |[x_0, \ldots, x_l]|^{d-l} \prod_{j=0}^{l}\|\mathrm{P}_E(n_P(x_j))\|^{-1} \lambda_{E \cap F_0 }(\mathrm{d}x_0)\ldots \lambda_{E \cap F_l}
     (\mathrm{d}x_l)\mu_{d,l}(\mathrm{d}E)  \\
     &+ (l!)^{d-l-1}b_{d-1,l}\cdot \sum_{F\in\mathcal F(P)} 
     \int\limits_{A_{F,l}}\int\limits_{(E \cap F)^{l+1}} h(x_0, \ldots, x_l)  \\
     &\times 
     |[x_0, \ldots, x_l]|^{d-l-1}\lambda_{E\cap F}(\mathrm{d}x_0)\ldots \lambda_{E \cap F}(\mathrm{d}x_l)\mu_{F,l}(\mathrm{d}E).
\end{align*}
This proves Theorem~\ref{thm4}.

\bibliographystyle{plain}
\bibliography{bib}

\end{document}